\documentclass[a4paper,11pt]{amsart}

\usepackage{amsmath,amssymb,amsfonts,latexsym}

\newtheorem{theorem}{Theorem}[section]

\newtheorem{Definition}[theorem]{Definition}

\newtheorem{Example}[theorem]{Example}
\begin{document}

\title[On The Topology of $n$-Normed Spaces ...]{On The Topology of $n$-Normed Spaces with Respect To Norms of Its Quotient Spaces}

\author[Harmanus Batkunde]{Harmanus Batkunde$^{1}$}
\address{$^{1}$Analysis and Geometry Group, Faculty of Mathematics and Natural Sciences Bandung Institute of Technology. Jl. Ganesha 10, Bandung 40132}
\email{h.batkunde@fmipa.unpatti.ac.id}

\author[Hendra Gunawan]{Hendra Gunawan$^{2}$}
\address{$^{2}$Analysis and Geometry Group, Faculty of Mathematics and Natural Sciences Bandung Institute of Technology. Jl. Ganesha 10 Bandung 40132}
\email{hgunawan@math.itb.ac.id}

%%%%%%%%%%%%%%%%%%%%%%%%%%%%%%%%%%%%%%%%%%%%%%%%%%%%%%%%%%%%%%%%%%%

\begin{abstract}
In this paper, we study some topological characteristics of the $n$-normed spaces. 
We observe convergence sequences, closed sets, and bounded sets in the $n$-normed 
spaces using norms of quotient spaces that will be constructed. These norms will 
be a new viewpoint in observing the characteristics of the $n$-normed spaces. By 
using these norms, we also review the completeness of the $n$-normed spaces.

\vspace{2mm}


\vspace{2mm}


\end{abstract}

\thanks{This work was supported by ITB research and inovation program 2018. ~The first author is supported by LPDP Indonesia.}

%%%%%%%%%%%%%%%%%%%%%%%%%%%%%%%%%%%%%%%%%%%%%%%%%%%%%%%%%%%%%%%%%%%

\maketitle

%%%%%%%%%%%%%%%%%%%%%%%%%%%%%%%%%%%%%%%%%%%%%%%%%%%%%%%%%%%%%%%%%%%

\section {Introduction}

%%%%%%%%%%%%%%%%%%%%%%%%%%%%%%%%%%%%%%%%%%%%%%%%%%%%%%%%%%%%%%%%%%%

In 1960's, the concept of $n$-normed spaces for $n \geq 2$ was introduced by G\"ahler \cite{5, 6, 7, 8}. 
This concept is studied and developed further by various researchers \cite{1, 2, 4, 10}. They also 
studied some characteristics of these spaces. Ekariani \textit{et al}. \cite{3} provided a contractive 
mapping theorem for the $n$-normed space of $p$-summable sequence. They proved the theorem using a 
linearly independent set consisting of $n$ vectors. Moreover, Gunawan \textit{et al}. \cite{12} proved 
fixed point theorems on bounded sets in an $n$-normed space also using a linearly independent set of $n$ 
vectors. Meitei and Singh \cite{13} studied bounded $n$-linear operators in $n$-normed spaces using similar approach. 
As we can see, most researchers usually investigate the $n$-normed spaces using a set of $n$ linear independent vectors.

Moreover, we will consider some quotient spaces of the $n$-normed spaces. We will construct these quotient 
spaces with respect to a linearly independent set. We also define a norm in each quotient space. We use 
these norms to observe some characteristics of the $n$-normed spaces. Before we present our main results, 
here are some basic concepts of the $n$-normed spaces.

%\begin{Definition}
	Let $n$ be a nonnegative integer and $X$ is a real vector space with $dim(X) \geq n$. An $n$-norm is a function
	$\|\cdot,\dots,\cdot\| : X^{n} \to \mathbb{R}$ which satisfies the following conditions:
	\begin{itemize}
		\item [i.] $\|x_1,\cdots,x_n\| \geq 0 ;$ \\
		$\|x_1,\cdots,x_n\| = 0$ if and only if $x_1,\cdots,x_n$ linearly dependent.
		\item [ii.] $\|x_1,\cdots,x_n\|$ is invariant under permutation.
		\item [iii.] $\|\alpha x_{1},\dots,x_{k}\|=|\alpha| \ \|x_{1},\dots,
		x_{k}\|$ for any $\alpha \in \mathbb{R}$.
		\item [iv.] $\|x_1+x_1^\prime,x_{2},\dots,x_{k}\| \leq \|x_1,x_{2},
		\dots,x_{k}\| + \|x_1^\prime,x_{2},\dots,x_{k}\|$.
	\end{itemize}
	The pair $\left(X,\|\cdot,\dots,\cdot\|\right)$
	is called an $n$-normed space \cite{10}.
	
%\end{Definition}

For example, if $(X,\langle \cdot,\cdot\rangle)$ is an inner product
space, we can define the standard $n$-norm on $X$ by
\begin{equation}\label{a}
\| x_{1},\dots,x_{n}\|^{S}:= \left\vert
\begin{array}{ccc}
\langle x_{1},x_1\rangle & \cdots & \langle x_{1},x_n\rangle \\
\vdots & \ddots & \vdots \\
\langle x_{n},x_1\rangle & \cdots & \langle x_{n},x_n\rangle%
\end{array}\right\vert^{1/2}.
\end{equation}
The determinant on the right hand of equation (\ref{a}) is known as Gram's determinant.
Its value is always nonnegative.~Geometrically, the value of $\|x_1,\dots,x_n\|^{S}$ also represents the volume of the
$n$-dimensional parallelepiped spanned by $x_1,\dots,x_n$.
Moreover, if $(X,\|\cdot,\dots,\cdot\|)$ is an $n$-normed space, then
$$\|x_{1},\dots,x_{n}\|=\|x_{1}+\alpha_{2} x_{2}+\dots+\alpha_{n}x_{n},\dots,x_{n}\|,$$
for all $x_{1},\dots,x_{n}\in X$ and $\alpha_{2},\dots,\alpha_{n} \in \mathbb{R}$ \cite{10}.

Now, we present the construction of quotient spaces of the $n$-normed spaces with respect to 
a linearly independent set. Let $(X,\|\cdot,\dots,\cdot\|)$ be an $n$-normed space and $Y=\{y_{1},\dots,y_{n}\}$ is a 
linearly independent set in $X$. For a $j \in \{1,\dots,n\}$ we consider $Y\setminus \{y_{j}\}$. We define a subspace of $X$,
namely
$$
Y_{j}^{0}:= {\rm span} \, Y \setminus \{y_{j}\} :=\left\{\sum_{\substack{i=1\\ i \neq j}}^{n} \alpha_{i} y_{i} 
\, \, ;  \, \,\alpha_{i} \in \mathbb{R}.\right\}.
$$
For any $u \in X$, the corresponding coset in $X$ is
$$\overline{u}:=\left\{u+\sum_{\substack{i=1\\ i\neq j}}^{n} \alpha_{i} y_{i}: \alpha_{i} \in \mathbb{R}\right\}.$$
Then we have $\overline{0} = {\rm span} \,Y \setminus \{y_{j}\} = Y_{j}^{0}$ and if $\overline{u} = \overline {v}$ 
then $u-v \in {\rm span} \, Y \setminus \{y_{j}\}$. We define quotient spaces of $X$ as
\begin{equation}\label{1}
X^{*}_{j}:=X/Y_{j}^{0}:=\{\overline{u}: u \in X\}.
\end{equation}
The addition and scalar multiplication apply in this space.
Moreover, we define a norm  of $X^{*}_{j}$ defined by
\begin{align} \label{2}
\begin{aligned}
\|\overline{u}\|^{*}_{j}:=& \|u,y_{1},\dots,y_{j-1},y_{j+1},\dots,y_{n}\|.
\end{aligned}
\end{align}
Using the above construction, we can get $n$ quotient spaces. These quotient spaces have same structure but different elements. 
The set that contains all quotient spaces constructed above is called class-$1$ collection.

In the above construction we 'eliminate' one vector from $Y$. Now we construct quotient spaces by 'eliminating' $m$ 
vectors from $Y$ with  $1\leq m \leq n$. Let $(X,\|\cdot,\dots,\cdot\|)$ be an $n$-normed space and 
$Y=\{y_{1},\dots,y_{n}\}$ is a linearly independent set in $X$. For an $m\in \mathbb{N}, 1 \leq m \leq n$ 
and $i_{1},\dots,i_{m} \in \{1,\dots,n\}$ we observe $Y\setminus \{y_{i_{1}},\dots,y_{i_{m}}\}$. We define a subspace of $X$,
namely
$$
Y_{i_{1},\dots, i_{m}}^{0}:= {\rm span} \, Y \setminus \{y_{i_{1}},\dots, y_{i_{m}}\} :=\left\{\sum_{\substack{i=1\\ i 
\neq i_{1},\dots, i_{m}}}^{n} \alpha_{i} y_{i} \, \, ;  \, \,\alpha_{i} \in \mathbb{R}.\right\}.
$$
For any $u \in X$, the corresponding coset in $X$ is
$$\overline{u}:=\left\{u+\sum_{\substack{i=1\\ i\neq i_{1},\dots, i_{m}}}^{n} \alpha_{i} y_{i}: \alpha_{i} \in \mathbb{R}\right\}.$$
Then we have $\overline{0} = {\rm span} \,Y \setminus \{y_{i_{1}},\dots, y_{i_{m}}\} = Y_{i_{1},\dots, i_{m}}^{0}$ 
and, if $\overline{u} = \overline {v}$, then $u-v \in {\rm span} \, Y \setminus \{y_{i_{1}},\dots, y_{i_{m}}\}$. 
We define quotient spaces of $X$ as
\begin{equation}\label{1}
X^{*}_{i_{1},\dots, i_{m}}:=X/Y_{i_{1},\dots, i_{m}}^{0}:=\{\overline{u}: u \in X\}.
\end{equation}
The addition and scalar multiplication apply in this space.
Moreover, we define a norm  of $X^{*}_{i_{1},\dots, i_{m}}$ defined by
\begin{align} \label{2}
\begin{aligned}
\|\overline{u}\|^{*}_{i_{1},\dots, i_{m}}:=& \|u,y_{1},\dots,y_{i_{1}-1},y_{i_{1}+1},\dots,y_{n}\|+\dots \\ 
&+ \|u,y_{1},\dots,y_{i_{m}-1},y_{i_{m}+1},\dots,y_{n}\|.
\end{aligned}
\end{align}
Using the above construction, we can get $\binom{n}{m}$ quotient spaces. For an $m  \in \mathbb{N}$ with 
$1 \leq m \leq n$, the set that contains all quotient spaces constructed above is called class-$m$ collection.

One can see that the right hand of equation (\ref{2}) is a summation of norms of class-$1$ collection. 
Then equation (\ref{2}) can be written as
$$\|\overline{u}\|^{*}_{i_{1},\dots, i_{m}}=\|\overline{u}\|^{*}_{i_{1}}+\dots + \|\overline{u}\|^{*}_{i_{m}}.$$
 Moreover, by saying 'norms of class-$m$ collection' we mean 'all the norms of each quotient space in a class-$m$ 
 collection'. Actually, some previous researchers who worked on $n$-normed spaces have used this approach. 
 They used the norm of class-$n$ collection to investigate some characteristics of the $n$-normed spaces. 
 If we compare our approach to their works, ours provides more than one viewpoint to observe characteristics 
 of the $n$-normed spaces. One can see that the summation of all norms of class-$1$ collection is the norm 
 of class-$n$ collection. Furthermore, for any $m \in \mathbb{N}, \,\, 1\leq m \leq n$ we will use the norms 
 of any class-$m$ collection to study some topological structures of the $n$-normed spaces.

%%%%%%%%%%%%%%%%%%%%%%%%%%%%%%%%%%%%%%%%%%%%%%%%%%%%%%%%%%%%%%%%%%%

\section{Results and Discussion}

%%%%%%%%%%%%%%%%%%%%%%%%%%%%%%%%%%%%%%%%%%%%%%%%%%%%%%%%%%%%%%%%%%%
As mentioned before, the construction of quotient spaces depends on a set of $n$ linearly independent vectors we choose. 
The choice of the set would not matter. We can choose any $n$ linearly independent vectors in the $n$-normed spaces to 
form the set. So from here on, we will not mention the linearly independent set explicitly, unless it is necessary.

Here we introduce some topological characteristics we observe by using the norms of class-$m$ collection with $m \in 
\mathbb{N}, 1 \leq m \leq n$.

\begin{Definition}\label{3}
	Let $(X,\|\cdot,\dots,\cdot\|)$ be an $n$-normed space. For an $m \in \mathbb{N}$ with $1\leq m \leq n$, we say 
a sequence $\{x_{k}\} \subset X$ converges with respect to the norms of class-$m$ collection to $x$ if for any 
$\epsilon >0 $, there exists an $N\in \mathbb{N}$ such that for $k \geq N$ we have
	$$\|\overline{x_{k}-x}\|^{*}_{i_{1},\dots,i_{m}} < \epsilon,$$
	for every $i_{1},\dots,i_{m} \in \{1,\dots,n\}$ with  $i_{1} < \cdots < i_{m}$. In this case we also say
	$$\lim_{n \to \infty} \|\overline{x_{k}-x}\|^{*}_{i_{1},\dots,i_{m}} = 0.$$
	for every $i_{1},\dots,i_{m} \in \{1,\dots,n\}$ with  $i_{1} < \cdots < i_{m}$. If $\{x_{k}\}$ does not converge, we say it diverges.
\end{Definition}

Note that this definition also says that a sequence converges with respect to the norms of a class-$m$ collection 
to $x$ if and only if the sequence converges in each quotient space of the class-$m$ collection.

We want to know if there is a relation between convergence with respect to the norms of class-$1$ collection and 
class-$m$ collection for any $m \in \mathbb{N}$ with $1\leq m \leq n$. First, we examine the following example.
\begin{Example}\label{5c}
	Let $(\mathbb{R}^{d},\|\cdot,\cdot,\cdot\|)$ be a $3$-normed space and for $m=1, 2, 3$ consider all norms in  
each class-$m$ collection. A sequence $\{x_{k}\}$ converges with respect to the norms of class-$2$ collection to 
$x$ if and only if, for any $\epsilon >0$ there exists an $N \in \mathbb{N}$ such that for any $k \geq N$ we have
	\begin{equation}\label{5}
	\begin{aligned}
	\|\overline{x_{k}-x}\|^{*}_{1,2} =  \|\overline{x_{k}-x}\|^{*}_{1} + \|\overline{x_{k}-x}\|^{*}_{2} < \frac{2}{3} \epsilon , \\
	\|\overline{x_{k}-x}\|^{*}_{1,3} = \|\overline{x_{k}-x}\|^{*}_{1} + \|\overline{x_{k}-x}\|^{*}_{3} < \frac{2}{3} \epsilon , \\
	\|\overline{x_{k}-x}\|^{*}_{2,3}  = \|\overline{x_{k}-x}\|^{*}_{2} + \|\overline{x_{k}-x}\|^{*}_{3} < \frac{2}{3} \epsilon .
	\end{aligned}
	\end{equation}
	Equation (\ref{5}) implies
	\begin{align}\label{5a}
	\|\overline{x_{k}-x}\|^{*}_{1,2,3}=\|\overline{x_{k}-x}\|^{*}_{1} + \|\overline{x_{k}-x}\|^{*}_{2}+ \|\overline{x_{k}-x}\|^{*}_{3} < \epsilon.
	\end{align}
	The left hand of equation (\ref{5a}) is nothing but the norm of class-$3$ collection. It means at the same time 
$\{x_{k}\}$ converges with respect to the norms of class-$3$ collection to $x$. Equation (\ref{5}) also implies
	\begin{align}\label{5b}
	\|\overline{x_{k}-x}\|^{*}_{1} < \epsilon \, ; \, \, \, \;  \|\overline{x_{k}-x}\|^{*}_{2} < \epsilon \, ;  
\, \, \, \; \|\overline{x_{k}-x}\|^{*}_{3} < \epsilon.
	\end{align}
	The left hands of three equations in (\ref{5b}) are the norms of class-$1$ collection. By definition, 
this means the sequence $\{x_{k}\}$ converges with respect to the norms of class-$1$ collection to $x$.
	Here we have if $\{x_{k}\}$ converges with respect to the norms of class-$2$ collection to $x$, then it 
converges with respect to the norms of class-$1$ collection to $x$. Also it converges with respect to the norms 
of class-$3$ collection to $x$. One can see that the converses are also true.
\end{Example}	
Example \ref{5c} indicates that all types of convergence are equivalent. This is true as we state in the following theorem.

\begin{theorem}\label{6}
	Let  $(X,\|\cdot, \dots, \cdot\|)$ be an $n$-normed space with class-$m$ collections of $X$ for any 
$m\in \mathbb{N}, 1\leq m \leq n$. A sequence $\{x_{k}\} \subset X$ is convergent with respect to the norms 
of class-$1$ collection if and only if it is convergent with respect to the norms of class-$m$ collection.
\end{theorem}

\begin{proof}
Let $(X,\|\cdot, \dots, \cdot\|)$ be an $n$-normed space and consider all class-$m$ collections of $X$. 
Suppose that  $\{x_{k}\} \subset X$ converges with respect to the norms of class-$1$ collection to $x$. 
Then for any $\epsilon > 0$ there exists an $N_{0} \in \mathbb{N}$ such that for $n\geq N_{0}$ we have
\begin{align*}
\|\overline{x_{k}-x}\|^{*}_{1} &< \frac{1}{m} \epsilon, \\
&\vdots \\
\|\overline{x_{k}-x}\|^{*}_{n} &< \frac{1}{m} \epsilon,
\end{align*}
for any $m\in \mathbb{N}, 1\leq m \leq n$. Therefore, we have
\begin{align*}
\|\overline{x_{k}-x}\|^{*}_{i_{1}}+\dots +\|\overline{x_{k}-x}\|^{*}_{i_{m}} < \epsilon,
\end{align*}
for every $i_{1},\dots, i_{m} \in \{1,\dots,n\}$. By the definition, this means $\{x_{k}\}$ converges 
with respect to the norms of class-$m$ collection to $x$.

Conversely, suppose that $\{x_{k}\}$ converges with respect to the norms of class-$m$ collection to $x$. 
Then for any $\epsilon > 0$ there exists an $N_{1} \in \mathbb{N}$ such that for $n\geq N_{1}$ we have
$$\|\overline{x_{k}-x}\|^{*}_{i_{1}}+\dots +\|\overline{x_{k}-x}\|^{*}_{i_{m}} < \epsilon,$$
for every $i_{1},\dots, i_{m} \in \{1,\dots,n\}$ and $i_{1}<\dots < i_{m}$. Then we have
$$\|\overline{x_{k}-x}\|^{*}_{j} < \epsilon, $$
for any $j \in \{1,\dots, n\}$. By the definition $\{x_{k}\}$ also converges with respect to the norms of class-$1$ collection to $x$.
\end{proof}

Theorem (\ref{6}) states that for any $1\leq m \leq n$, all types of convergence with respect to class-$m$ 
collection are equivalent. So unless we need to specify the class explicitly, we may simply use the word 
'converges' instead of 'converges with respect to the norms of class-$m$ collection'. We also denote $\{x_{k}\}$ 
converges to $x$ by $x_{n} \longrightarrow x$. %To make it simpler we define an class set by $$I_{m,n}:= \{\{i_{1},\dots,i_{m}\}: i_{1},\cdots,i_{m} \in \{1,\dots,n\} \text { with } i_{1}< \dots < i_{m} \}.$$
Here is a definition of a closed set in the $n$-normed space.
\begin{Definition}\label{4}
	Let $(X,\|\cdot,\dots,\cdot\|)$ be an $n$-normed spaces and $K \subseteq X$. The set $K$ is called 
closed if for any sequence $\{x_{k}\}$ in $K$ that converge in $X$, its limit belongs to $K$.
\end{Definition}

One can see that for observing the convergence of a sequence we have to use all norms of a class-$m$ collection. 
For simplicity, we need to reduce the number of the norms we used. So, is it possible to reduce the number of the norms? 
First, we consider the following example.

\begin{Example}
	Let $(\mathbb{R}^{d}, \|\cdot, \cdot, \cdot\|)$ be a $3$-normed space. We define class-$1$,2,3 collections of $\mathbb{R}^{d}$.
	Suppose that a sequence $\{x_{k}\} \subset \mathbb{R}^{d}$ convergence with respect to the norms of class-$2$ collection to $x$. 
For any $\epsilon > 0$, there is an $N \in \mathbb{N}$ such that for $k \geq N$ we have
	$$\|\overline{x_{k}-x}\|^{*}_{1,2} < \epsilon,$$
	$$\|\overline{x_{k}-x}\|^{*}_{1,3} < \epsilon,$$
	$$\|\overline{x_{k}-x}\|^{*}_{2,3} < \epsilon.$$
	Now, let's just consider the first two equations. The two equations imply
	$$ \|\overline{x_{k}-x}\|^{*}_{1}  < \epsilon, \, \, \, \, \, \, \,
	\|\overline{x_{k}-x}\|^{*}_{2}  < \epsilon, \, \, \, \, \, \, \,  \|\overline{x_{k}-x}\|^{*}_{3}  < \epsilon.$$
	By the definition, we have $\{x_{k}\}$ converges with respect to the norms of class-$1$ collection to $x$. 
Since all types of convergence are equivalent for any class-$m$ collection (with $m \in \mathbb{N}, 1\leq m \leq n$), 
it is sufficient to review the convergence just by using two norms of class-$2$ collection, namely 
$\|\cdot\|_{1,2}$ and $\|\cdot\|_{1,3}$.	
\end{Example}

\noindent \textbf{Remark.} Generally, for a fix $ m\in \mathbb{N}, 1 \leq m \leq n,$ the convergence (of a sequence) 
with respect to the norms of class-$m$ collection is sufficient to be observed by using some norms 
$\|\cdot\|^{*}_{i_{1},\dots,i_{m}}$ we choose such that
\begin{align}\label{10}
\bigcup \{i_{1},\dots,i_{m}\} \supseteq \{ 1,\dots,n\}.
%\text{ with } I\subset I_{m,n}
\end{align}
Moreover, the least number of norms that can be used to define convergence of class-$m$ is 
$\left\lceil \frac{n}{m}\right\rceil$ norms. For some class-$m$ collections, we do not need to use all norms of a class-$m$ collection to observe the convergence. One can see that %that there are $n$ quotient spaces in class-$1$ collection and only one quotient space in class-$n$ collection. This means that
for $m=1$ or $m=n$ we have to use all norms of the class-$m$ collection.

For example, let $(X,\cdot,\cdot,\cdot,\cdot,\cdot|\|)$ be a $5$-normed space and consider the norms 
of class-$2$ collection. The least number of norms that we can use is $\left\lceil \frac{5}{2}\right\rceil = 3$, 
namely $\|\cdot\|_{1,2}^{*}, \, \, \|\cdot\|_{1,3}^{*}, \, \|\cdot\|_{4,5}^{*}$.
If it is less than three norms, then it will not meet the definition requirements. This following example 
shows what happened if we choose less norms than $\left\lceil \frac{n}{m}\right\rceil$ norms.

\begin{Example}
	Let $(\mathbb{R}^{5},\langle\cdot,\cdot\rangle)$ be an inner product space and $Y$ is a set of 
standard basis vectors of $\mathbb{R}^{5}$. We define a standard $5$-norm as
	\begin{equation}\label{1p2}
	\| x_{1},\dots,x_{5}\|^{S}:= \left\vert
	\begin{array}{ccc}
	\langle x_{1},x_{1}\rangle & \cdots & \langle x_{1},x_{5}\rangle \\
	\vdots & \ddots & \vdots \\
	\langle x_{5},x_{1}\rangle & \cdots & \langle x_{5},x_{5}\rangle%
	\end{array}\right\vert^{1/2}.
	\end{equation}
	We consider sequence $x_{k}=(0,0,0,0,k)$ in $\mathbb{R}^{5}$. By the definition, one can see that 
this sequence $x_{k}$ is not converge with respect to the norms of class-$m$ collection, for any $m=1,\dots,5$.

Let us consider class-$2$ collection of $\mathbb{R}^{5}$ and observe the convergence. Note that three is 
the least number of norms we propose to observe the convergence with respect to the norms of class-$2$ collection. 
Instead of using three norms, here we try to use two norms of class-$2$ collection. Let $x=(0,0,0,0,0)$, 
choose two norms of class-$2$ collection, namely $\|\cdot\|^{*}_{1,2}, \|\cdot\|^{*}_{3,4}$.
	For any $k \in \mathbb{N}$ we have
	\begin{equation*}
	\|\overline{x_{k}-x}\|^{*}_{1,2}=\|\overline{x_{k}}\|^{*}_{1,2}=\|x_{k},y_{2},y_{3},y_{4},y_{5}\|^{S} + 
\|x_{k},y_{1},y_{3},y_{4},y_{5}\|^{S} =0,
	\end{equation*}
	and
	\begin{equation*}
	\|\overline{x_{k}-x}\|^{*}_{3,4}=  \|\overline{x_{k}}\|^{*}_{1,2}=\|x_{k},y_{1},y_{2},y_{4},y_{5}\|^{S} + 
\|x_{k},y_{1},y_{2},y_{3},y_{5}\|^{S} =0.
	\end{equation*}
	the value of these norms is 0 because $x_{k} = k\, y_{5}$, which means $x_{k}$ and $y_{5}$ are linearly dependent. 
If we just take these two norms to investigate the convergence of $\{x_{k}\}$, then we have %for any $\epsilon > 0$ there is an $N \in \mathbb{N}$ such that for $k \geq N$ implies that
	%$\|\overline{x_{k}-x}\|^{*}_{1,2}< \epsilon$ and $\|\overline{x_{k}-x}\|^{*}_{3,4} < \epsilon$, which means
	$x_{n} \longrightarrow x$. This is a false conclusion. %It will also happened if we choose any two norms of class-$2$ collection to observe the convergence of the sequence.
	But if we add another norm of class-$2$, for example $ \|\overline{x_{k}-x}\|^{*}_{1,5}$, then we have
	\begin{align*}
	\|\overline{x_{k}-x}\|^{*}_{1,5} & =  \|x_{k},y_{2},y_{3},y_{4},y_{5}\|^{S} + \|x_{k},y_{1},y_{2},y_{3},y_{4}\|^{S} \\
	& =0 + \left\vert
	\begin{array}{ccccc}
	\langle x_{k},x_{k}\rangle & \langle x_{k},x_{1}\rangle &\langle x_{k},x_{2}\rangle &\langle x_{k},x_{3}\rangle & \langle x_{k},y_{4}\rangle \\
	\langle x_{1},x_{k}\rangle & \langle x_{1},x_{1}\rangle &\langle x_{1},x_{2}\rangle &\langle x_{1},x_{3}\rangle & \langle x_{1},y_{4}\rangle \\
	\langle x_{2},x_{k}\rangle & \langle x_{2},x_{1}\rangle &\langle x_{2},x_{2}\rangle &\langle x_{2},x_{3}\rangle & \langle x_{2},y_{4}\rangle \\
	\langle x_{3},x_{k}\rangle & \langle x_{3},x_{1}\rangle &\langle x_{3},x_{2}\rangle &\langle x_{3},x_{3}\rangle & \langle x_{3},y_{4}\rangle \\
	\langle y_{4},x_{k}\rangle & \langle x_{4},x_{1}\rangle &\langle x_{4},x_{2}\rangle &\langle x_{4},x_{3}\rangle & \langle y_{4},y_{4}\rangle
	\end{array}\right\vert^{1/2}\\
	%	& =0 + \left\vert
	%	\begin{array}{ccccc}
	%	n^{2} & 0 & 0 & 0 & 0 \\
	%	0  & 1 & 0 & 0 & 0 \\
	%	0 & 0 & 1 & 0 & 0 \\
	%	0 & 0 & 0 & 1 & 0 \\
	%	0 & 0 & 0 & 0 & 1
	%	\end{array}\right\vert^{1/2}.\\
	&=n.
	\end{align*}	
	Then, for $n \to \infty$, $\|\overline{x_{k}-x}\|^{*}_{1,5} \to \infty$, which means that $x_{k}$ diverges.
\end{Example}

The example shows that in this case, we can not take less than $\left\lceil \frac{5}{2} \right\rceil$ norms 
to examine the convergence of a sequence. Also, The norms we choose must satisfy condition (\ref{10}). 
Furthermore, we will study bounded sets in $n$-normed spaces with respect to the norms of quotient sets of the $n$-normed spaces.
\begin{Definition}
	Let $(X,\|\cdot, \dots, \cdot\|)$ be an $n$-normed spaces with class-$m$ collection for an $m \in \mathbb{N}, 
1\leq m \leq n$ and $K \subseteq X$ be a nonempty set. The set $K$ is called bounded with respect to the norms of 
class-$m$ collection if and only if for any $x \in K$ there exists an $M>0$ such that
	$$\|\overline{x}\|^{*}_{i_{1},\dots,i_{m}}\leq M,$$
	for every $i_{1},\dots,i_{m} \in \{1,\dots,n\}$ with $i_{1}<\dots<i_{m}$.
\end{Definition}

Moreover, we will show that all types of the boundedness with respect to the norms of class-$m$ collections for 
any $m \in \mathbb{N}, 1 \leq m \leq n$ are equivalent. Let's consider the following example.

\begin{Example}
	Let $(\mathbb{R}^{d},\|\cdot,\cdot,\cdot\|)$ be a $3$-normed space, $K \subset \mathbb{R}$ and define class-$1, 2, 3$ 
collections in $\mathbb{R}^{d}$.
	Suppose that $K \subseteq \mathbb{R}^{d}$ is bounded with respect to the norms of class-$2$ collection. Then for any 
$x \in K$, there is an $M>0$ such that
	\begin{equation}
	\begin{aligned}
	\|\overline{x}\|^{*}_{2,3}  = \|\overline{x}\|^{*}_{2}+\|\overline{x}\|^{*}_{3} \leq M, \\
	\|\overline{x}\|^{*}_{1,3} = \|\overline{x}\|^{*}_{1}+\|\overline{x}\|^{*}_{3} \leq M,\\
	\|\overline{x}\|^{*}_{1,2} = \|\overline{x}\|^{*}_{1}+\|\overline{x}\|^{*}_{2} \leq M.
	\end{aligned}
	\end{equation}
	These imply
	$$2\left(\|\overline{x}\|^{*}_{1}+\|\overline{x}\|^{*}_{2}+\|\overline{x}\|^{*}_{3} \right)=
\|\overline{x}\|^{*}_{2,3}+\|\overline{x}\|^{*}_{1,3}+\|\overline{x}\|^{*}_{1,2}  \leq 3M.$$
	or
	\begin{align}\label{7} \|\overline{x}\|^{*}_{1,2,3}=\|\overline{x}\|^{*}_{1}+
\|\overline{x}\|^{*}_{2}+\|\overline{x}\|^{*}_{3} \leq \frac{3}{2}M =C,
	\end{align}
	with $C>0$. One can see that the left hand of equation (\ref{7}) is the norm of class-$3$ collection. 
This means that $K$ is bounded with respect to the norm in class-$3$ collection.
	Equation (\ref{7}) implies
	\begin{align}\label{7b}
	\|\overline{x}\|^{*}_{1} \leq C, \quad \|\overline{x}\|^{*}_{2}\leq C, \quad \|\overline{x}\|^{*}_{3} \leq C.
	\end{align}
	The left hands of the three equations in (\ref{7b}) are nothing but all norms of class-$1$ collection. 
This means $K$ is bounded with respect to the norms of class-$1$ collection.

In this example, if $K$ is bounded with respect to the norms of class-$2$ collection then $K$ is bounded 
with respect to the norms of class-$3$ collection. It also implies $K$ is bounded with respect to the norms 
of class-$1$ collection. One can see that the converses are also true. The example indicates that all types 
of boundedness are equivalent. We present the equivalence of the boundedness in the following theorem.
\end{Example}

\begin{theorem}\label{8}
	Let  $(X,\|\cdot, \dots, \cdot\|)$ is an $n$-normed space with class-$m$ collections for any 
$m \in \mathbb{N}, 1 \leq m \leq n$ and $K \subset X$ nonempty. The set $K$ is bounded with respect 
to the norms of class-$1$ collection if and only if it is bounded with respect to class-$m$ collection.
\end{theorem}

\textit{Proof.}
Let $(X,\|\cdot, \dots, \cdot\|)$ is an $n$-normed space, $K \subset X$ nonempty and consider all 
class-$m$ collections of $X$, for $m\in \mathbb{N}, 1 \leq m \leq n$. Suppose that the set $K$ is 
bounded with respect to the norms of class-$1$ collection, then for any $x\in K$, there is an  $M>0$ such that
$$\|\overline{x}\|^{*}_{1} \leq M, \quad \|\overline{x}\|^{*}_{2}\leq M, \quad \cdots , \quad \|\overline{x}\|^{*}_{n} \leq M.$$
Then for an $m \in \mathbb{N}, 1\leq m \leq n$ we have
\begin{align}\label{9}
\|\overline{x}\|^{*}_{i_{1},\dots,i_{m}}=\|\overline{x}\|^{*}_{i_{1}} + \dots + 
\|\overline{x}\|^{*}_{i_{m}} \leq m \cdot M =C, \quad C>0,
\end{align}
for any $i_{1},\dots,i_{m} \in \{1,\dots,n\}$. The left hand of equation (\ref{9}) 
represents all norms of class-$m$ collection. By the definition, $K$ is bounded with 
respect to the norms of class-$m$ collection with $m \in \mathbb{N}, 1\leq m \leq n$.

Conversely, supposed that $K$ is bounded with respect to the norms of class-$m$ collection for 
an $m\in \mathbb{N}, 1 \leq m \leq n$. For any $x \in K$ there exists an $M>0$ such that
$$\|\overline{x}\|^{*}_{i_{1}} + \dots + \|\overline{x}\|^{*}_{i_{m}}=\|\overline{x}\|^{*}_{i_{1},\dots,i_{m}} \leq M.$$
for every $i_{1},\dots,i_{m} \in \{1,\dots,n\}$ with $i_{1}<\dots<i_{m}$. It implies
$$\|\overline{x}\|^{*}_{j}\leq M, $$  for every $j\in \{1,\dots, n\}$. Consider the norms 
of class-$1$ collection. Then by the definition this means $K$ is bounded with respect to the 
norms of class-$1$ collection. \qed

Since all types of boundedness are equivalent, we will not mention the type of boundedness explicitly. 
We may simply use the word 'bounded' instead of 'bounded with respect 
to the norms of class-$m$ collection'.

We will give an example to show that we do not need to choose all norms of class-$m$ collection for some 
$m \in \mathbb{N}, 1 \leq m \leq n$ to examine the boundedness.

\begin{Example}
	We observe $(\mathbb{R}^{d},\|\cdot,\cdot,\cdot\|)$ as a $3$-normed space and consider all norms of 
class-$2$ collection. We have three norms in class-$2$ collection, namely $\|\cdot\|_{1,2}^{*},
\|\cdot\|_{1,3}^{*},\|\cdot\|_{2,3}^{*}$. It is sufficient to examine the boundedness of a set 
$K \subseteq \mathbb{R}^{d}$ just by using two norms of class-$2$, namely $\|\overline{x}\|^{*}_{1,2}, 
\|\overline{x}\|^{*}_{1,3}$. If, for every $x \in K$, 
	$$\|\overline{x}\|^{*}_{1,2} \leq M, \text{ and } \, \, \|\overline{x}\|^{*}_{1,3}\leq M; \, \, M>0, $$
	then we have
	$$2\|\overline{x}\|^{*}_{1}+\|\overline{x}\|^{*}_{2}+\|\overline{x}\|^{*}_{3} \leq 2M = C; \, \, C>0.$$
	This implies
	\begin{equation}\label{8a}
	\|\overline{x}\|^{*}_{1,2,3} \leq 2\|\overline{x}\|^{*}_{1}+\|\overline{x}\|^{*}_{2}+\|\overline{x}\|^{*}_{3} \leq C.
	\end{equation}
	The left hand of equation (\ref{8a}) is the norm of class-$3$ collection. By the definition, this means $K$ 
is bounded with respect to the norms of class-$3$ collection. Since all types of boundedness are equivalent, $K$ is bounded.	
\end{Example}

Generally, the definition of bounded set of class-$m$ is sufficient to be observed by just using some norms 
$\|\cdot\|^{*}_{i_{1},\dots,i_{m}}$ such that condition (\ref{10}) applies.
Moreover, the sufficient number of norms that we use to define bounded with respect to the norms of class-$m$ 
collection is $\left\lceil \frac{n}{m}\right\rceil$ norms. One can see that these conditions match with 
conditions for convergence of a sequence.

Next, we will present a concept of completeness in the $n$-normed spaces with respect to the norms of class-$m$ collection. 
Here we give some basic definitions related to completeness.

\begin{Definition}
	Let $(X,\|\cdot,\dots,\cdot\|)$ be an $n$-normed spaces. For any $m\in \mathbb{N}, 1\leq m \leq n$ 
a sequence $\{x_{k}\} \subset X$ is called a Cauchy sequence with respect to the norms of class-$m$ 
collection if for any $\epsilon >0$, there exists an $N\in \mathbb{N}$ such that, for every $k,l \geq N$, we have
	$$\|\overline{x_{k}-x_{l}}\|^{*}_{i_{1},\dots,i_{m}}< \epsilon,$$
	for every  $i_{1},\dots,i_{m} \in \{1,\dots,n\}$ and $i_{1} < \dots < i_{m}$. In other words
	$$\lim_{k,l \to \infty} \|x_{k}-x_{l}\|^{*}_{i_{1},\dots,i_{m}}=0,$$
	for every  $i_{1},\dots,i_{m} \in \{1,\dots,n\}$ and $i_{1} < \dots < i_{m}$.
\end{Definition}

\begin{theorem}
	Let $(X,\|\cdot,\dots,\cdot\|)$ be an $n$-normed spaces with class-$m$ collections for any 
$m\in \mathbb{N}, 1\leq m \leq n$.  If $\{x_{n}\}$ is convergent with respect to the norms of class-$m$ collection, 
then $\{x_{n}\}$ is Cauchy with respect to the norms of class-$m$ collection.
\end{theorem}

\begin{proof}
	Let an $m\in \mathbb{N}, 1\leq m \leq n$ and $\{x_{n}\}$ converges with respect to class-$m$ collection to $x$. 
Then we have for any $\epsilon >0$, there exists an $N\in \mathbb{N}$ such that, for every $k,l \geq N$ we have
	$$\|\overline{x_{k}-x_{l}}\|^{*}_{i_{1},\dots,i_{m}} \leq  \|\overline{x_{k}-x}\|^{*}_{i_{1},\dots,i_{m}} + 
\|\overline{x_{l}-x}\|^{*}_{i_{1},\dots,i_{m}} < \epsilon,$$
	for every  $i_{1},\dots,i_{m} \in \{1,\dots,n\}$ and $i_{1} < \dots < i_{m}$. This means $\{x_{n}\}$ is a Cauchy sequence.	
\end{proof}

Moreover, let $(X,\|\cdot,\dots,\cdot\|)$ be an $n$-normed spaces with class-$m$ collections for any $m\in \mathbb{N}, 1\leq m \leq n$.
Furthermore, all types of Cauchy sequence with respect to the norms of class-$m$ collection for any $m\in \mathbb{N}, 1 \leq m \leq n$
are equivalent. We state it in the following theorem.

\begin{theorem}
	Let $(X,\|\cdot,\dots,\cdot\|)$ be an $n$-normed spaces with class-$m$ collections for any $m\in \mathbb{N}, 1\leq m \leq n$. 
A sequence $\{x_{n}\} \subset X$ is Cauchy with respect to the norms of class-$1$ collection if and only if 
$x_{n}$ is Cauchy with respect to the norms of class-$m$ collection.
\end{theorem}

\begin{proof}
The proof is analogous with the proof of Theorem (\ref{6}).
\end{proof}

Hereafter, we may simply use the word 'Cauchy' instead of 'Cauchy with respect to the norms of class-$m$ collection'. 
We will mention the type if we need to specify it explicitly. Moreover, if every Cauchy sequence in $X$ is convergent, 
then $X$ is complete.

%%%%%%%%%%%%%%%%%%%%%%%%%%%%%%%%%%%%%%%%%%%%%%%%%%%%%%%%%%%%%%%%
\section{Remark and Conslusion}

We have observed some topological characteristics of the $n$-normed spaces. We studied closed sets, bounded sets, 
convergence and Cauchy sequences. We also studied the completeness of the $n$-normed spaces.~We reviewed these 
characteristics using norms of the quotient spaces we constructed from the $n$-normed spaces. Some researchers have 
used this approach.~They used the norms of class-$n$ collection to observe some characteristics of the $n$-normed spaces. 
Here, we provide more than one class collection that can be used to see the characteristics of the $n$-normed spaces. 
This means we use more general viewpoints. This allows us to present more general definitions, properties and theorems 
as part of the characteristics of the $n$-normed spaces. Using this viewpoint, we can investigate not only the 
topological characteristics, but also geometry and other aspects of the $n$-normed spaces.
%%%%%%%%%%%%%%%%%%%%%%%%%%%%%%%%%%%%%%%%%%%%%%%%%%%%%%%%%%%%%%%%

%%%%%%%%%%%%%%%%%%%%%%%%%%%%%%%%%%%%%%%%%%%%%%%%%%%%%%%%%%%%%%%%%%%

\end{document}